\documentclass[a4paper,reqno]{amsart}

\usepackage{amsmath}
\usepackage{mathrsfs,amssymb,stmaryrd,bbm}
\usepackage[all,2cell]{xy}
\xyoption{v2}
\UseAllTwocells
\usepackage[showonlyrefs]{mathtools}

\theoremstyle{plain}
\newtheorem{thm}{Theorem}
\newtheorem{cor}[thm]{Corollary}
\newtheorem{lemma}[thm]{Lemma}
\newtheorem{prop}[thm]{Proposition}

\theoremstyle{remark}

\newtheorem{rmk}[thm]{Remark}

\newtheorem{ex}[thm]{Example}

\theoremstyle{definition}

\newtheorem{df}{Definition}

\newcommand{\Ind}{\operatorname{Ind}}
\newcommand{\op}{{\mathrm{op}}}

\newcommand{\fin}{\ensuremath{\operatorname{Fin}}}

\newcommand{\Cat}{\ensuremath{\mathbf{Cat}}}

\newcommand{\V}{\ensuremath{V}}
\newcommand{\VCat}{\ensuremath{\V\text{-}\Cat}}

\newcommand{\delten}{\ensuremath{\bullet}}

\newcommand{\Rex}{\operatorname{Rex}}
\DeclareMathOperator*\colim{colim}
\newcommand{\kk}{\ensuremath{k}}
\newcommand{\kmod}{\ensuremath{\kk\text{-}\mathbf{Mod}}}

\newcommand{\Lex}{\ensuremath{\operatorname{Lex}}}

\begin{document}
\title[Tensor products of categories]{Tensor products of finitely cocomplete and abelian
  categories}
\author{Ignacio L\'opez Franco}
\address{Computing Department and Department of Mathematics\\
  Macquarie University, NSW 2109\\Australia}
\email{ignacio.lopezfranco@mq.edu.au}
\date{\today}
\thanks{The author acknowledges the support of the Australian Research
  Council in the form of the Discovery Project DP1094883 and of a
  Research Fellowship, Gonville and  Caius College, University of Cambridge. 
}
\begin{abstract}
The purpose of this article is to study the existence of Deligne's tensor
product of abelian categories by comparing it with the well-known
tensor product of finitely cocomplete categories. The main result
states that the former exists precisely when the latter is an abelian
category, and moreover in this case both tensor products
coincide. An example of two abelian categories whose Deligne tensor
product does not exist is given. 
\end{abstract}
\keywords{Deligne tensor product, abelian category, colimit}
\maketitle{}
\section{Introduction}



In recent years Deligne's tensor product of abelian categories has
encountered new 
applications apart from the original \cite{Deligne:TannakianCats},
especially in works on {\em tensor categories}\/ and their actions on
categories, in relation to Hopf algebra theory
\cite{Etingof:FinTenCats,Etingof:AnalogueRadford,Etiongof:OnFusionCats},
but also in derived geometry \cite{Lyubashenko:TenProdEqSh},
reconstruction theorems \cite{Lyubashenko:SqrHopf,1211.3678} and topological
quantum field theories \cite{MR1862634}. 
Despite its growing use, the theoretical material on this tensor product
still reduces to little more than the original definition.

\begin{df}[\cite{Deligne:TannakianCats}]\label{df:1}
  Deligne's tensor product of two abelian \kk-linear categories $A,B$
  (\kk\ a commutative ring) 
  is an {\em abelian} category  $A\delten B$, with a bilinear functor
$A\times B\to A\delten B$ that is right exact in each variable
and induces equivalences
\begin{equation}\label{eq:7}
\Rex[A\delten B,C]\simeq \Rex[A,B;C]
\end{equation} 
for each abelian 
category $C$. The category on the left hand side of \eqref{eq:7} is
the usual category of right exact $\kk$-linear functors and natural
transformations, while the one on the right hand side is the category
of $\kk$-bilinear functors $A\times B\to C$ that are right exact in each
variable. 
Observe that this determines $A\delten{}B$  only up to a unique up to
isomorphism equivalence of categories. 
\end{df} 

Despite the abstract definition, \cite{Deligne:TannakianCats} states
the existence of the tensor product only for a certain class of
abelian categories ---and we complete the proof of that result.
The monograph \cite{Lyubashenko:SqrHopf}, that 
uses the tensor product of \cite{Deligne:TannakianCats} and explores
in some length the resulting monoidal 2-category, does not
concern itself 
with the existence of the tensor product.

In the present paper we relate Deligne's tensor product with another
tensor product that extends it. The 
reason to consider another tensor product is one that some readers
may have already realised: Definition \ref{df:1} is peculiar in
that it requires $A\delten B$ to be an abelian category while its
universal property speaks solely about right exact functors.
This mismatch between the structures present on the categories and the
structures preserved by the functors make the existence of an
universal object as $A\delten B$ questionable. A more natural
definition would be one in which the categories are only finitely
cocomplete, dropping any assumption of abelianness.


Tensor products of categories with a certain class of colimits have
been described in
\cite{Kelly:BCECT,Kelly:BCECTrep}. An early instance of this
construction, for algebraic theories, appeared in \cite{MR0262334}.
The present article argues
in favour of the tensor product of categories with finite
colimits as an alternative to Deligne's tensor product.


Let us now recall what is shown in \cite{Deligne:TannakianCats} about
the existence of the tensor product of abelian categories. In a
similar fashion to Definition \ref{df:1}, 
\cite[5.1]{Deligne:TannakianCats} defines the tensor product of a
family of \kk-linear abelian categories, where \kk\ is a commutative
ring. Here a  first problem arises: the product of the empty
family of abelian categories ---or equivalently the unit object for the
tensor product if one prefers Definition \ref{df:1}--- has to be
necessarily equivalent to the category of finitely presentable
\kk-modules. This category is abelian if and only if \kk\ is a
coherent ring \cite{MR0120260,MR0217051}.
Therefore, if we want to have a unit object for the
tensor product we must assume that the ground commutative ring is
coherent. In practise this is not too restrictive a condition, but we
encounter more problems in binary products.

As we already mentioned, \cite{Deligne:TannakianCats} does not prove the
existence of its tensor product of abelian categories in general, and
we shall see in Section \ref{sec:exist-delign-tens} that there are
examples where indeed it does not exist. The existence result
\cite[5.13]{Deligne:TannakianCats} states that tensor
products do exist when the categories are \kk-linear abelian for \kk\ a field,
have finite dimensional homs and objects of finite length. As pointed
out by Deligne in a private communication, the proof in
\cite{Deligne:TannakianCats} is incomplete, and therefore we provide a
full proof.  
 
With the aim of making this paper accessible to as many readers as
possible, henceforth we shall assume that our categories are enriched
in $\V=\kmod$, for a fixed commutative ring \kk. 
Sometimes we use the term \kk-linear categories.
Functors will also be enriched in \V, automatically making ordinary
natural transformations enriched in \V. From time to time, when the
results hold in greater generality we shall make an explicit remark.    

The paper is organised as follows. Section \ref{sec:background}
compiles some of the prerequisites for the rest of the article,
including free completions under colimits,
bicolimits of pseudofunctors and tensor
products of finitely cocomplete categories. 
In Section \ref{sec:abel-categ-mult} we obtain conditions that
ensure that categories of left exact functors in two variables are
abelian. These results are employed 
in Section
\ref{sec:exist-delign-tens}, which deal with the question of the existence
of Deligne's tensor product of abelian categories by comparing it with
the tensor product of finitely cocomplete categories. We show
\begin{thm}
  \label{thm:5}
  Given two small abelian categories $A,B$ the following are equivalent.
  \begin{enumerate}
  \item
    Deligne's tensor product of two abelian categories $A,B$ exists.
  \item
    The tensor product of $A,B$ as finitely cocomplete categories is
    an abelian category.
  \end{enumerate}
\end{thm}
Using this result we provide an example, based on work by Soublin
\cite{MR0260799}, where \kk\ is a field:
\begin{cor}
  \label{cor:3}
  There exist two abelian \kk-linear categories whose Deligne's tensor
  product does not exist. 
\end{cor}
We give complete the proof of the result in \cite{Deligne:TannakianCats}:
\begin{thm}
  \label{thm:7}
  For a field \kk, the
   tensor product (as finitely cocomplete categories) of two \kk-linear abelian
   categories with finite dimensional homs and objects
   of finite length is again
   abelian. In particular their Deligne tensor product exists. 
 \end{thm}
Finally, Section \ref{sec:semis-categ} treats the case
of simisimple categories by characterising them as special free
completions under finite colimits. 
 
The author would like to thank Martin Hyland for this support
insightful conversations, Mike Prest for pointing out Soublin's
results on coherent rings and Pierre Deligne for invaluable
communications and spotting a number of mistakes in a draft version of
the present paper.

\section{Background}\label{sec:background}
This section collects known results and constructions necessary to
develop the rest of the paper. 

All categories and functors will be \kk-linear, {\em i.e.}, enriched in
the category $\V=\kk\text-\mathrm{Mod}$ of \kk-modules for a
commutative ring \kk. 
The tensor product of two such categories $A,B$ is a \V-category
$A\otimes B$ with objects $\operatorname{ob}A\times\operatorname{ob}B$
and enriched homs $(A\otimes B)((a,b),(a',b'))=A(a,a')\otimes
B(b,b')$. Identities and composition are given in the obvious way
---using the symmetry of \V. A functor of two variables from $A,B$ to
a third category $C$ is a functor $A\otimes B\to C$.

\subsection{Locally finitely presentable categories}
\label{sec:locally-finit-pres}

From time to time we shall need to use some classical facts about
locally finitely presentable categories, which for the sake of
clarity of exposition we present in this separate section. Locally
finitely presentable categories were introduced in \cite{MR0327863},
and full expositions can be found in \cite{MR1031717} and
\cite{MR1294136}, while their theory in the context of enriched categories was
developed in \cite{Kelly:StructuresFiniteLimits}. We refer the reader
to these sources for more background on the facts we recall below. 
Our case is
somewhere between the classical and the enriched situations in the
sense that, although all our categories are enriched in
$\V=\kk\text-\mathrm{Mod}$, because this category is extremely well
behaved all the nuances of the enriched context disappear.

We say that a full subcategory $K:G\hookrightarrow A$ is a {\em strong
generator\/} (resp. {\em dense\/}) if the functor $A\to[G^\op,\V]$ given by $a\mapsto
G(K-,a)$ reflects isomorphisms (resp. is fully faithful).

An object $a$ of the cocomplete category $A$ is {\em finitely
  presentable}\/ when the representable functor $A(a,-):A\to\V$
preserves filtered colimits. 
The full subcategory of locally
presentable objects is denoted by $A_f$. Because filtered colimits
distribute over finite limits in \V,
it is easy to see that $A_f$ is closed in $A$ under finite colimits. 

\begin{thm} 
  \label{thm:8}
  For a \V-category $A$ the following are equivalent.
  \begin{enumerate}
  \item\label{item:11}
    $A$ is cocomplete and has a small strong generator formed by
    finitely presentable objects.
  \item\label{item:12}
    $A$ is equivalent to $\operatorname{Lex}[C^\op,\V]$ for some small
    finitely cocomplete category $C$. 
  \end{enumerate}
  If a category $A$ satisfies \ref{item:11} then the functor $A\to
  \operatorname{Lex}[A_f^\op,\V]$ induced by the inclusion
  $A_f\subset A$ is an equivalence. Moreover if
  $G\subset A_f$ is the strong generator, $A_f$ is the closure
  of $G$ under finite colimits. 
  Conversely,
  $\operatorname{Lex}[C^\op,\V]_f$ is given by the representable
  presheaves, and thus equivalent to $C$. 
\end{thm}

A \V-category satisfying these equivalent conditions is called a
{\em locally finitely presentable}\/ \V-category. Any such category is
complete.

The following result will be used in Lemma \ref{l:1}. In order for a functor
$F:A\to B$ between locally finitely presentable categories that
preserves filtered colimits to be left exact it is enough to preserve
kernels of arrows between finitely presentable objects. This is a
consequence of the distributivity of filtered colimits over finite
limits in any locally finitely presentable category, and a short proof
can be produced by using the {\em uniformity lemma}\/ of
\cite{MR1004603}. 

\subsection{Completions under colimits}
\label{sec:compl-under-colim}

By a class of conical colimits $\Phi$ we mean a class of small categories.
We say that a category $A$ has $\Phi$-colimits, or it is
$\Phi$-cocomplete, if every functor $f:D\to A$ with
$D\in\Phi$ has a (conical) colimit. 
A functor is $\Phi$-cocontinuous when it preserves $\Phi$-colimits. 

Given a class of conical colimits $\Phi$, a free completion of a
category $X$ under $\Phi$-colimits is a functor $X\to \Phi(X)$, where
$\Phi(X)$ is a $\Phi$-cocomplete category, that under
composition induces an equivalence
\begin{equation*}
  \Phi\text-\mathrm{Cocts}(\Phi(X),A)\xrightarrow{\simeq}[X,A]
\end{equation*}
between the categories of $\Phi$-cocontinuous functors $\Phi(X)\to A$
and the category of functors $X\to A$, for all $\Phi$-cocomplete
categories $A$. This universal property determines $\Phi(X)$ only up
to a unique up to isomorphism
equivalence of categories

The free completion $\Phi(X)$ can be constructed as the closure under
$\Phi$-colimits of the representable functors in $[X^\op,\V]$. That
is the smallest replete full subcategory closed
under $\Phi$-colimits and containing the
representables. 
The universal $X\to \Phi(X)$ is the co-restriction of the Yoneda
embedding. See \cite[Section 5.7]{Kelly:BCECT}.


\subsubsection{Finite colimits}
\label{sec:finite-colimits}

Of special interest for us is the class of finite colimits \fin,
formed by the finite categories. A finite colimit is exactly a
\fin-colimit. (In a setting of
categories enriched in a more general 
category \V\ usual conical colimits are not enough and one is lead to
consider weighted colimits \cite[Ch. 3]{Kelly:BCECT}
---previously known as indexed colimits.)
The completion of $X$ under finite colimits is the
smallest replete full
subcategory $\fin(X)\subset[X^\op,\V]$ closed under finite colimits
and containing the representables. Since the category $[X^\op,\V]$
is locally finitely presentable and the representable presheaves form
a strong generator, $\fin(X)$ is the full subcategory of
its 
finitely presentable objects, {\em i.e.}, those $\phi$ such that
$[X^\op,\V](\phi,-)$ preserves filtered colimits (see Theorem
\ref{thm:8} and
\cite{MR0327863,Kelly:StructuresFiniteLimits}).  
For example, if $R$ is a
\kk-algebra and $\Sigma R$ is the associated one object \kk-category,
then $\fin(\Sigma R)$ is equivalent to $R\text-\mathrm{Mod}_f$, the
category of finitely presentable $R$-modules.

Observe that the full subcategory $X\hookrightarrow\fin(X)$ is dense (and in
particular a strong generator) and consists of projective
objects. Conversely, if a finitely cocomplete category $A$ has a
strong generating full subcategory $X\subset A$ consisting of
projective objects then $A\simeq\fin(X)$.

The following lemma, which will be needed later, holds because 
the category \V\ is extremely well behaved, in our case a category of
modules over a commutative ring. The proof is along the lines of
\cite[8.11]{Kelly:StructuresFiniteLimits}.

\begin{lemma}\label{l:8}
  For any small category $X$, any $\phi\in\fin(X)\subset [X^\op,\V]$ can
  be obtained as the coequaliser of a pair of
  arrows between finite coproducts of representable functors.
\end{lemma}



\subsubsection{Filtered colimits}
\label{sec:filtered-colimits}

Another special case of free completion under colimits we shall use is
the one corresponding to filtered colimits. If $X$ is a small category
we write $\Ind X$ for its free completion under filtered colimits and
$X\to \Ind X$ the universal functor. This is equivalent to the
category of ind-objects in \cite[I.8.2]{SGA4}.  We summarise some
results on completions of finitely cocomplete categories under filtered
colimits in the following theorem (see
\cite{MR0327863,Kelly:StructuresFiniteLimits}). 

\begin{thm}\label{thm:10}
  If $X$ has finite 
  colimits then:
  \begin{enumerate}
  \item\label{item:13}
    Its completion under filtered colimits can be obtained as
    the co-restriction of the Yoneda embedding
    \begin{equation}\label{eq:18}
      X\longrightarrow{}\operatorname{Lex}[X^\op,\V]
    \end{equation}
  \item\label{item:14}
    $\Ind X$ is locally finitely finitely presentable (see Theorem \ref{thm:8} and
    \cite{MR0327863,Kelly:StructuresFiniteLimits}).
  \item\label{item:15}
    The universal functor $X\to \Ind X$ preserves finite colimits.
  \item\label{item:16}
    The full subcategory $(\Ind X)_f$ of finitely presentable objects of $\Ind
    X$ is the replete image of \eqref{eq:18}, and thus equivalent to
    $X$. 
  \item\label{item:17}
    The functor \eqref{eq:18} induces an equivalence
    \begin{equation*}
      \mathrm{Cocts}[\Ind X,Y]\simeq \mathrm{Rex}[X,Y]
    \end{equation*}
    for all cocomplete categories $Y$ between the category of
    cocontinuous functors $\Ind X\to Y$ and the one of right exact
    functors $X\to Y$. 
  \end{enumerate}
  Conversely, for any locally finitely presentable category $A$ the
  inclusion $A_f\subset A$ induces an equivalence $A\simeq
  \Ind{A_f}$. 
\end{thm}




\subsection{Bicolimits and biadjunctions}
\label{sec:bicolimits}

In Section \ref{sec:an-existence-result} we shall need the notion of a {\em bicolimit}\/ of a
pseudofunctor, which is a weakening of the usual definition of
colimit. Bicategories, pseudofunctors, pseudonatural transformations
and modifications between them where introduced in
\cite{Benabou.bicat}. All our bicategories will be strict or 2-categories:
the category of small \V-categories, \V-functors and \V-natural
transformations \VCat; 
the 2-category of finitely cocomplete \V-categories, right exact
\V-functors and \V-natural transformations $\mathbf{Rex}$;
the 2-category of abelian categories, exact functors and transformations.

A pseudofunctor $F:\mathscr K\to\mathscr L$ between 2-categories
preserves composition and identities only up to coherent
isomorphisms. Similarly, a pseudonatural transformation $\sigma:F\Rightarrow
G:\mathscr K\to\mathscr L$ is given by components $\sigma_X:FX\to GX$
and for each $f:X\to Y$ in $\mathscr K$ an isomorphism
$\sigma_Y.(Ff)\cong (Gf).\sigma_X$ satisfying coherence
conditions. There is a notion of morphism between pseudonatural
transformations, or modification. We refer the reader to the original
\cite{Benabou.bicat}, or for a more recent reference
\cite{pre05659661}. 

Bi(co)limits appear in \cite{fibrations} in the full
generality of bicategories ---see also
\cite{2categoricallimits}. Bicolimits of pseudofunctors $J\to\Cat$
where $J$ is a category already appear in \cite{SGA4}. 
If $F:\mathscr K\to \mathscr L$ is a pseudofunctor, a bicolimit of $F$ is a
category $\operatorname{bicolim}F$ together with a pseudonatural
transformation into a constant pseudofunctor
$F\Rightarrow \Delta\operatorname{bicolim}F$ that induces equivalences
$\mathscr L(\operatorname{bicolim}F,Z)\simeq \mathbf{Ps}(F,\Delta Z)$
where $\Delta Z$ is the constant pseudofunctor on the object $Z$ and
the category on the right hand side is the category of pseudonatural
transformations and modifications from $F$ to $\Delta$.

A biadjunction between two pseudofunctors $F:\mathscr
K\rightleftarrows \mathscr L$ is an equivalence $\mathscr
L(FX,Z)\simeq\mathscr K(X,GZ)$, pseudonatural in $X,Z$. We say that 
$F$ is a left biadjoint of $G$ and $G$ a right biadjoint of $F$.

If $F$ has a right biadjoint, then it preserves any bicolimit that may
exist in its domain, a fact that will be used in the proof of
Proposition \ref{prop:4}.

\subsection{Tensor product of finitely cocomplete categories}
\label{sec:tens-prod-finit}
This section gives a construction of a tensor product of
finitely cocomplete categories, following
\cite{Kelly:StructuresFiniteLimits,Kelly:BCECTrep}.  The base category
\V\ can be any locally finitely presentable symmetric monoidal 
closed category, in which case one must employ
weighted colimits. However, for our purposes \V\ will be \kmod.

\begin{thm}[\cite{Kelly:StructuresFiniteLimits}]
  \label{thm:9}
  Given two small finitely cocomplete categories $A,B$, there is a
  another $A\boxtimes B$ and a functor $\chi:A\otimes B\to A\boxtimes B$ that
  is right exact in each variable and universal with
  this property, in the following sense. For each finitely cocomplete
  category $C$, composition with $\chi$ induces an
  equivalence
  \begin{equation}\label{eq:2}
    \Rex[A\boxtimes B,C]\xrightarrow{\simeq}\Rex[A,B;C]\cong\Rex[A,\Rex[B,C]]
  \end{equation}
  between the category of right exact functors $A\boxtimes B\to C$ and
  the category of functors $A\otimes B\to C$ right exact in each
  variable. This universal property determines $A\boxtimes B$ up to a
  unique up to a unique isomorphism equivalence.
\end{thm}

The finitely cocomplete categories $\Rex[A,B]$ make the 2-category of
finitely cocomplete categories, right exact functors and natural
transformations into a {\em pseudo-closed 2-category}\/
\cite{Hyland:PsdComm}. The tensor product $\boxtimes$ is part of a
monoidal 2-category structure in the weak sense of
\cite{tricategories,McCrudden:BalancedCoalg}. The unit object of this
monoidal structure is the category $\V_f$ of finitely presentable
objects of $\V$: we have canonical equivalences $\V_f\boxtimes A\simeq
A\simeq A\boxtimes V_f$. 

Below we recall a construction of $A\boxtimes B$ provided in
\cite{Kelly:BCECT,Kelly:StructuresFiniteLimits} but do not include the
proof of the theorem, which is lengthily and better understood in a
slightly more general context.

The full
subcategory $\operatorname{Lex}[A^\op,B^\op,\V]$ of $[(A\otimes
B)^\op,\V]$ of functors $A^\op\otimes B^\op \to \V$ left exact in each
variable is reflexive ---{\em i.e.}, the inclusion functor has a
left adjoint--- \cite[Theorem 6.11]{Kelly:BCECT}. In particular
$\operatorname{Lex}[A^\op,B^\op;\V]$ is complete and
cocomplete. Moreover, the reflection is fully faithful if and only if
$\operatorname{Lex}[A^\op,B^\op;\V]$ contains the representables if
and only if 
for each $a\in A,b\in B$ the functors $A\to A\otimes B$, $B\to
A\otimes B$ given by $x\mapsto (x,b)$ and $y\mapsto (a,y)$
preserve finite colimits. 

One obtains a functor
\begin{equation}
  \label{eq:10}
  A\otimes B\xrightarrow{Y} [(A\otimes B)^\op,\V]\xrightarrow{R} \operatorname{Lex}[A^\op,B^\op;\V]
\end{equation}
composition of the Yoneda embedding and the reflection. The following
result can be deduced from \cite[Ch. 6]{Kelly:BCECT} but we provide a
direct proof. 

\begin{lemma}
  \label{l:11}
  The functor \eqref{eq:10} is dense and right exact in each
  variable. 
\end{lemma}
\begin{proof}
  Let us write $L=\operatorname{Lex}[A^\op,B^\op;\V]$. 
  The density of \eqref{eq:10} follows from
  \cite[Prop. 5.7]{Kelly:BCECT} since $R$ is a left Kan extension of
  $RY$ along $Y$ and all reflections are dense.

  The functor $RY$ of \eqref{eq:10} is right exact in each variable if and only
  if $L(RY-,\phi):(A\otimes B)^\op\to\V$ is left exact in each
  variable for each $\phi\in L$. But
  \begin{equation*}
    L(RY-,\phi)\cong [(A\otimes  B)^\op,\V](Y-,\phi)\cong\phi \in L
  \end{equation*}
  where the first isomorphism is the one of the reflection and the
  second is the one provided by the Yoneda lemma. It follows that $RY$
  is right exact in each variable. 
\end{proof}

Define
$A\boxtimes B$ as the closure under finite colimits of the image of the
functor \eqref{eq:10}; this means that $A\boxtimes B$ is the smallest
replete ---{\em i.e.}, full and closed under isomorphisms--- subcategory
of $\operatorname{Lex}[A^\op,B^\op;\V]$ containing the image of
\eqref{eq:10} and closed under finite colimits. The functor
\begin{equation}\label{eq:11}
  \chi:  A\otimes B\to A\boxtimes B
\end{equation}
obtained by co-restriction of \eqref{eq:10} preserves finite colimits
in each variable and satisfies the universal property of Theorem \ref{thm:9}.
Explicitly, given $F:A\otimes B\to C$ right
exact in each variable the corresponding
right exact functor $A\boxtimes B\to C$ is a left Kan extension of $F$
along $A\otimes B\to A\boxtimes B$. This universal functor is dense,
and hence so is the full subcategory of $A\boxtimes B$ consisting of
objects of the form $\chi(a,b)$ ---by \cite[Thm. 5.13]{Kelly:BCECT}.

\begin{rmk}
  \label{rmk:6}
  The universal functor \eqref{eq:11} 
  is fully faithful if and only if the functors $A\to
  A\otimes B$, $B\to A\otimes B$ given by $x\mapsto (x,b)$ and
  $y\mapsto (a,y)$ respectively preserve finite colimits. In
  particular, it is fully faithful if either of the following holds:
  \begin{itemize}
  \item
  The hom-objects of $A,B$ are
  flat objects in \V. For example, when the tensor product of \V\
  preserves finite limits ({\em e.g.}, $\V=\mathbf{Vect}$).
  \item
    Finite colimits in $A,B$ are
    absolute, {\em i.e.}, preserved by any functor. For example, if
    $A,B$ are semisimple categories.  
  \end{itemize}
\end{rmk}

\begin{rmk}\label{rmk:4}
  Being cocomplete with a small dense subcategory of finitely
  presentable objects, 
  $\operatorname{Lex}[A^\op,B^\op;\V]$ is locally finitely
  presentable ---see Section \ref{sec:locally-finit-pres}. Its full
  subcategory of finitely presentable objects is $A\boxtimes B$, by
  Theorem \ref{thm:8}.
\end{rmk}


\begin{ex}\label{ex:1}
  Given the categories of finitely presentable modules of two
  $\kk$-algebras $\mathsf A,\mathsf B$, the functor 
  \begin{equation*}
    \mathsf{A}\text-\mathrm{Mod}_f\otimes \mathsf{B}\text-\mathrm{Mod}_f
    \xrightarrow{\otimes_{\kk}}
    (\mathsf{A}\otimes\mathsf{B})\text-\mathrm{Mod}_f
  \end{equation*}
  satisfies the universal property \eqref{eq:2}, so
  \begin{equation*}
    \mathsf A\text-\mathrm{Mod}_f\boxtimes \mathsf
    B\text-\mathrm{Mod}_f\simeq
    (\mathsf A\otimes\mathsf B)\text-\mathrm{Mod}_f.
  \end{equation*}

\end{ex}

\begin{rmk}
  \label{rmk:1}
  It is a completely formal consequence of the definitions that there
  are canonical equivalences
  $\fin(X)\boxtimes\fin(Y)\simeq\fin(X\otimes Y)$ and
  $\V_f\simeq\fin(I)$, making $\fin:\VCat\to\mathbf{Rex}$ a monoidal
  pseudofunctor. 
\end{rmk}

The equivalences
\begin{equation}
  \label{eq:12}
  \Rex[A,\Rex[B,C]]\simeq\Rex[A\boxtimes B,C]\simeq\Rex[B,\Rex[A,C]]
\end{equation}
imply that $(-\boxtimes A),(A\boxtimes-):\mathbf{Rex}\to\mathbf{Rex}$
have a right biadjoint $\Rex[A,-]$. Therefore, by the comments at the
end of Section \ref{sec:bicolimits}, the tensor product $\boxtimes$
preserves bicolimits in each variable. 

\section{Abelian categories of multilinear lex functors}
\label{sec:abel-categ-mult}

The question of when the tensor product of categories with finite
colimits of Section \ref{sec:tens-prod-finit} restricts to abelian
categories, in the sense that $A\boxtimes B$ is abelian whenever $A,B$
are abelian, is central to the problem of the
existence of Deligne's tensor product explored in the next
section. For the time being, we consider the following weaker
question: is the free completion under filtered colimits of $A\boxtimes
B$,
\begin{equation*}
  \Ind(A\boxtimes B)\simeq\Lex[A^\op,B^\op;\V]
\end{equation*}
an abelian category if $A,B$ are abelian? This section answers this
question affirmatively by showing that there is a left exact reflection
\begin{equation}
  \label{eq:1}
  \xymatrix{
    \Lex[A^\op,B^\op;\V]\ar@{^(->}@<5pt>_-{\top}[r]&
    [(A\otimes B)^\op,\V]\ar@<5pt>[l]
    }
\end{equation}

Firstly we consider the category $\Lex[C^\op,\V]$, which is well-known
to be reflective in the presheaf category $[C^\op,\V]$ for a very
general \V ---as a
consequence of \cite[Thm. 6.5]{Kelly:BCECT} or \cite[9.7]{Kelly:StructuresFiniteLimits}. However,
because we are assuming that \V\ is a category of modules over a
commutative ring, we can use the even older result that for an abelian
small category $C$, the reflection
\begin{equation}
  \label{eq:3}
  L:[C^\op,\V]\longrightarrow{}\Lex[C^\op,\V]
\end{equation}
is {\em left exact}\/
\cite[Ch. 2]{MR0232821}. Alternatively, we could characterise left
exact \kk-linear functors $C^\op\to\V$ as sheaves for a \kk-linear
Grothendieck topology in $C$ and apply the results in
\cite{MR1394507}. During the preparation of this manuscript
\cite{1211.3678} appeared, where an elementary characterisation
of those finitely cocomplete \kk-linear
categories $C$ for which \eqref{eq:3} is left exact is given.

The proof of the following lemma, due to P. Freyd and posted on the categories
mailing list on the 30th of October 2005, is straightforward.

\begin{lemma}[P. Freyd]
  \label{l:9}
  In an abelian category any exact diagram
    \begin{equation}\label{eq:8}
      \diagram
      x\ar[r]^\lambda\ar[d]_\chi &
      y\ar[r]^\xi\ar[d]_\epsilon&
      q\ar[r]\ar[d]^\varphi&
      0\\
      x'\ar[r]^{\mu}&
      y'\ar[r]^\zeta&
      q'\ar[r]&
      0
      \enddiagram
    \end{equation}
    can be extended to an {\em exact}\/ solid diagram as exhibited in Figure
    \ref{fig:1}.
    \begin{figure}
    $$
    \xymatrixcolsep{1.5cm}
    \diagram
    &0\ar[r]&
    k'\ar[r]    \ar[d]&
    y+x'+x\ar[r]^-{\left(\begin{smallmatrix}1&\epsilon\\0&\mu\\\lambda&0\end{smallmatrix}\right)}
    \ar[d]_{\left(\begin{smallmatrix}1&0\\0&1\\0&0\end{smallmatrix}\right)}
    &
    y+y'\ar[d]_{\left(\begin{smallmatrix}0\\1\end{smallmatrix}\right)}\\
    &0\ar[r]& k\ar[r]\ar[d]&
    y+x'\ar[r]^-{\left(\begin{smallmatrix}\epsilon\\\mu\end{smallmatrix}\right)}
      \ar@{..>}[dl]^{\left(\begin{smallmatrix}\xi\\0\end{smallmatrix}\right)}&
    y'\\
    x\ar[r]^\lambda\ar[d]_\chi& y\ar[r]^\xi\ar[d]_\epsilon& q\ar[r]\ar[d]^\varphi&
    0&\\
    x'\ar[r]^{\mu}& y'\ar[r]^\zeta& q'\ar[r]& 0&
    \enddiagram
    $$
    \caption{}
    \label{fig:1}
  \end{figure}
\end{lemma}

\begin{lemma}
  \label{l:1}
  For a finitely cocomplete category $A$, $\Lex[C^\op,\V]$ is abelian
  if and only if the reflection \eqref{eq:3} is left exact.
\end{lemma}
\begin{proof}
  One direction is clear, namely if the reflection is left exact
  then $L(C)=\Lex[C^\op,\V]$ is abelian. For the converse, recall that
  $P(C)=[C^\op,\V]$ is 
  locally finitely presentable with finitely presentable objects given
  by the category $\fin(C)\hookrightarrow{}[C^\op,V]$. Therefore, the 
  reflection $L$ \eqref{eq:3} preserves kernels if and only if it
  preserves kernels of arrows between objects of $\fin(C)$ ---see the
  last paragraph of Section \ref{sec:locally-finit-pres}. By Lemma
  \ref{l:8}, any object of $\fin(C)$ is a cokernel of an arrow between
  representable functors. An arrow $q\to q'$ in $\fin(C)$ can be fitted into
  a diagram of the form \eqref{eq:8}, where $x,y,x',y'$ are
  representable presheaves, since representables are projective. Now we can use Lemma \ref{l:9} to obtain an exact
  diagram in $P(C)$ as in Figure \ref{fig:1}. In particular,
  $\operatorname{ker}(q\to q')\cong\operatorname{coker}(k'\to k)$.
  Observe that the two top rows in Figure \ref{fig:1} are kernels of
  morphisms between representable presheaves, and therefore the only
  presheaves that may not be left exact are $q,q'$.

  Now apply $L$ to
  the diagram of Figure \ref{fig:1} to obtain a similar one in $L(C)$ that we can see to be
  exact. Indeed, $L$ preserves the two cokernels and preserves two
  kernels because these are kernels in $L(C)$.
  Since $L(C)$ is
  abelian we deduce that 
  $\operatorname{ker}(L\varphi:Lq\to Lq')\cong\operatorname{coker}(Lk'\to Lk)$ in
  $L(C)$. Together with $\operatorname{coker}(Lk'\to Lk)\cong
  L\operatorname{coker}(k'\to k)\cong L\operatorname{ker}(\varphi :q\to q')$ we
  deduce that $L$ preserves the kernel of $\varphi$.
\end{proof}

\begin{cor}
  \label{cor:4}
  Let $A,B$ be finitely cocomplete categories and $\Lex[A^\op,B^\op;\V]$ the
  category of \V-functors $A^\op\otimes B^\op\to\V$ left exact in each
  variable. If both $\Lex[A^\op,\V]$ and $\Lex[B^\op,\V]$ are abelian
  then left adjoint to the inclusion
  \begin{equation}
    \label{eq:32}
    \Lex[A^\op,B^\op;\V]\subseteq[(A\otimes B)^\op,\V]
  \end{equation}
  is left exact and $\Lex[A^\op,B^\op;\V]$ is abelian. This holds, in
  particular, if $A,B$ are abelian. 
\end{cor}
\begin{proof}
  To save space, let us write
  ${L}(A)$, ${L}(B)$, ${L}(A,B)$ for
  $\Lex[A^\op,V]$, $\Lex[B^\op,\V]$, $\Lex[A^\op,B^\op;\V]$, and
  $i_A,i_B,i_{A,B}$ for the inclusion functors into the respective
  presheaf categories. The inclusion
  \begin{equation*}
    i_{A,B}:L(A,B)\hookrightarrow[(A\otimes  B)^\op,\V]
  \end{equation*}
  is, up to composing with the obvious isomorphisms
  $L(A,B)\cong\Lex[A^\op,L(B)]$, $[(A\otimes
  B)^\op,\V]\cong[A^\op,[B^\op,\V]]$, the following composition.
  \begin{equation}
    \label{eq:33}
    \Lex[A^\op,L(B)]\hookrightarrow[A^\op,L(B)]
    \xrightarrow{[A^\op,i_B]} [A^\op,[B^\op,\V]]
  \end{equation}
  By Lemma \ref{l:1}, the last arrow in \eqref{eq:33} has a left exact
  left adjoint,
  so it is enough to prove that the first arrow
  also has a left exact left adjoint. This inclusion functor is, up to
  composing with the isomorphisms $\Lex[A^\op,L(B)]\cong
  \Lex[B^\op,L(A)]$ and $[A^\op,L(B)]\cong\Lex[B^\op,[A^\op,\V]]$,
  \begin{equation*}
    \Lex[B^\op,L(A)]\xrightarrow{\Lex[B^\op,i_A]}\Lex[B^\op,[A^\op,\V]] 
  \end{equation*}
  which has a left exact left adjoint given by $\Lex[B^\op,L]$ where
  $L:[A^\op,\V]\to L(A)$ is the left exact left adjoint of $i_A$. 
\end{proof}

\begin{rmk}
  \label{rmk:5}
  If $C$ is finitely cocomplete and finitely complete,
  $L=\Lex[C^\op,V]$ is not necessarily abelian. In fact, for such a
  category $C$, if $L$ is  abelian the exact embedding $C\to L$ forces
  $C$ to be abelian. 
\end{rmk}

During the preparation of this manuscript \cite{1211.3678} appeared,
where Lemma \ref{l:1} is proven using different techniques
---Grothendieck topologies--- and Corollary \ref{cor:4} shown only for
\kk-linear categories over a {\em field}\/ \kk. 

\section{On the existence of Deligne's tensor product}
\label{sec:exist-delign-tens}
Now we can proceed to compare Deligne's tensor product of abelian
categories with the tensor product of finitely cocomplete categories
of previous sections. By doing so we are
able to provide an example of two abelian categories whose Deligne
tensor product does not exist. 

Given two abelian categories $A,B$, their Deligne's tensor product as
defined in \cite{Deligne:TannakianCats}
is an {\em abelian} category $A\delten B$ together with a \kk-linear
functor $A\otimes B\to A\delten B$, right 
exact in each variable, that induces equivalences $\Rex[A\delten
B,C]\simeq\Rex[A,B;C]$ for any {\em abelian\/} $C$.

\begin{lemma}
  \label{l:10}
  Suppose Deligne's tensor product of two abelian categories $A,B$
  exists. Then the functor
  \begin{equation}
    \label{eq:21}
    A\otimes B\longrightarrow A\delten{}B\longrightarrow{} \Ind(A\delten{}B)
  \end{equation}
  has the following universal property. For any
  cocomplete abelian category $C$, composition with \eqref{eq:21}
  induces an equivalence between the category of cocontinuous functors
  $\Ind(A\delten{}B)\to C$ and the category of functors $A\otimes
  B\to C$ right exact in each variable.
  \begin{equation}
    \label{eq:22}
    \mathrm{Cocts}[\Ind(A\delten{}B),C]\simeq \mathrm{Rex}[A,B;C]
  \end{equation}
\end{lemma}
\begin{proof}
  The functor \eqref{eq:22} is the composition of the following two
  equivalences. 
  $$
  \mathrm{Cocts}[\Ind(A\delten{}B),C]\longrightarrow
  \mathrm{Rex}[A\delten{}B,C]\longrightarrow
  \mathrm{Rex}[A,B;C]
  $$
  The first equivalence is the one provided by Theorem
  \ref{thm:10}~\ref{item:17}. 
\end{proof}

\begin{thm}
  \label{thm:3}
  Given two abelian categories $A,B$ the following are equivalent.
  \begin{enumerate}
  \item \label{item:9}
    Deligne's tensor product $A\delten B$ of $A,B$ exists.
  \item \label{item:10}
    The tensor product $A\boxtimes B$ of $A,B$ as categories with
    finite colimits is abelian.
  \end{enumerate}
\end{thm}
\begin{proof}
  Recall from Remark \ref{rmk:4} that $A\boxtimes B$ can be constructed
  as $L_f$, the full 
  subcategory of $L=\Lex[A^\op,B^\op;\V]$ of finitely presentable objects. Moreover $L$ is
  locally finitely presentable, so $\Ind{(L_f)}\simeq L$ by Theorem
  \ref{thm:10}. 
  Assume that $A\delten{}B$ exists. Its completion under filtered colimits
  $\Ind(A\delten{}B)$ is locally finitely
  presentable with $(\Ind(A\delten{}B))_f\simeq A\delten{}B$, and in
  addition it is an abelian category as any completion under
  filtered colimits of an abelian category is. It is
  enough, then, to prove that $L$ and $\Ind(A\delten{}B)$ are
  equivalent. To this end we show that the composition of the Yoneda
  embedding with the reflection into $L$
  \begin{equation}
    \label{eq:24}
    A\otimes B\longrightarrow{}[(A\otimes B)^\op,\V] \longrightarrow{} L
  \end{equation}
  has the universal
  property of $\Ind(A\delten{}B)$ established in Lemma
  \ref{l:10}. This functor always induces equivalences
  \begin{equation*}
    \mathrm{Cocts}[L,C]\simeq\mathrm{Rex}[A\boxtimes B,C]\simeq
    \mathrm{Rex}[A,B;C]
  \end{equation*}
  for all cocomplete categories $C$; the first equivalence is given by
  consequence of the equivalence $\Ind(A\boxtimes B)\simeq L$ and 
  Theorem \ref{thm:10}~\ref{item:17}.
  Together with the fact that $L$ is abelian (Corollary \ref{cor:4}),
  this shows that \eqref{eq:24} has the
  universal property of Lemma \ref{l:10}.

  Conversely, if $A\boxtimes B$ abelian, it clearly has the universal
  property of $A\delten B$. 
\end{proof}

\subsection{A counterexample to existence}
\label{sec:count-exit}

The question arises of whether there exists a pair of abelian
categories whose Deligne tensor product does not exist, or
equivalently, whose tensor product $\boxtimes$ is not abelian. We show:

\begin{cor}\label{cor:5}
  Given an arbitrary field \kk,
  there exists a pair of abelian \kk-linear categories whose
  Deligne's tensor product does not exist. 
\end{cor}

  Recall that given 
  $\kk$-algebras $\mathsf R,\mathsf S$
  \begin{equation*}
    \mathsf R\text-\mathrm{Mod}_f\boxtimes \mathsf S\text-\mathrm{Mod}_f
  \simeq (\mathsf R\otimes_{\kk} \mathsf S)\text-\mathrm{Mod}_f
  \end{equation*}
  as categories enriched in $\kk$-vector spaces (see Example
  \ref{ex:1}). 
  
  The category of finitely presented
  $\mathsf R$-modules $\mathsf R\text-\mathrm{Mod}_f$ is abelian
  precisely when $\mathsf R$ is
  a {\em left coherent}\/ ring. Another characterisation of coherence
  says that a ring $\mathsf R$ is left 
  coherent when every finitely generated left ideal is
  finitely presentable. For example, Noetherian rings are coherent. 
  The notion of a coherent ring was introduced in \cite{MR0120260} and
  \cite{MR0217051}. 
  

The following example is given in \cite[Prop. 19]{MR0260799}, and
together with Theorem \ref{thm:3} yields Corollary \ref{cor:5}. 

\begin{ex}
  \label{ex:2}
  Let \kk\ be a field and $\mathsf
  A\subset \kk^{\mathbb N}$ be the sub-\kk-algebra of stationary
  sequences. It is easy to check that $\mathsf A$ is von Neumann regular, and thus
  coherent, but the \kk-algebra of formal series $\mathsf A[[x]]$ is not
  coherent. This is shown in \cite[Prop. 19]{MR0260799}, for
  $\kk=\mathbb Q$, but the argument carries over an arbitrary
  field. Therefore we have two coherent \kk-algebras $\mathsf
  A,\kk[[x]]$ whose tensor product over \kk\ is not coherent. 
\end{ex}

\subsection{An existence result}
\label{sec:an-existence-result}

We finish the section with a positive result on the existence of Deligne's tensor
product of two abelian categories, namely a completion of the
proof of \cite[Prop. 5.13 (i)]{Deligne:TannakianCats}. 

Following the terminology of \cite{Lyubashenko:SqrHopf}, we say that
an abelian category $A$ over a field has {\em length}\/ if 
\begin{enumerate}
\item $A$ has finite dimensional homs.
\item All the objects of $A$ have finite length.
\end{enumerate}

\begin{prop}[\cite{Deligne:TannakianCats}]
  \label{prop:4}
  Let \kk\ be a field and $A,B$ abelian \kk-linear categories with
  length. Then their Deligne tensor product $A\delten{}B$ exists. 
\end{prop}


P. Deligne pointed out in a private communication to the author a
that the proof of  Proposition \ref{prop:4} (\cite[Prop
5.13]{Deligne:TannakianCats}) relies on an inaccurate statement at the
end of \cite[5.1]{Deligne:TannakianCats}.
The rest of this section is devoted to give a proof of this result.
We shall need the following lemmas.

\begin{lemma}[\cite{Deligne:TannakianCats}]
  \label{l:7}
  Any abelian category $A$ with length over a field is filtered union of
  full subcategories $A_\alpha\subset A$ closed under subobjects and quotients, each of which is equivalent
  to $\mathsf A_\alpha\text-\mathrm{Mod}_f$ for some finite
  dimensional algebra $\mathsf A$.
\end{lemma}
 
\begin{lemma}\label{l:4}
  Let $\mathsf{A}$ be a finite dimensional algebra and $K:B\hookrightarrow
  \mathsf{A}\text-\mathrm{Mod}_f$ the 
  inclusion of a full subcategory closed under direct sums, subobjects
  and quotients. Then $B$ is equivalent to
  $\mathsf{B}\text-\mathrm{Mod}_f$ for some finite dimensional algebra
  $\mathsf{B}$ and $K$ is, up to composing with this equivalence, the
  restriction of scalars functor induced by a morphism of algebras
  $\mathsf{B}\to \mathsf{A}$. 
\end{lemma}
\begin{proof}
  The functor $K$ has a left
  adjoint $K^\ell$ defined by $K^\ell\mathsf{M}$= biggest quotient
  of $\mathsf{M}$ that lies in $B$. Since $K$ is right exact,
  $K^\ell\mathsf{A}$ is projective and 
  moreover it is a generator (strong generator) because $K$ is faithful
  (reflects isomorphisms). Therefore $B$ is equivalent to
  $\mathsf{B}\text-\mathrm{Mod}_f$ where
  $\mathsf{B}=B(K^\ell\mathsf{A},K^\ell\mathsf{A})$, the equivalence
  given by $b\mapsto B(K^\ell\mathsf A,b)$.
  Now one can check that the functor $\mathsf
  B\text-\mathrm{Mod}_f\to\mathsf A\text-\mathrm{Mod}_f$ is isomorphic
  to the restriction of scalars functor along
  \begin{equation}
    \label{eq:13}
    \mathsf A\cong\mathsf A\text-\mathrm{Mod}_f(\mathsf A,\mathsf A)
    \xrightarrow{K^\ell}B(K^\ell\mathsf A,K^\ell\mathsf A)=\mathsf B
  \end{equation}
\end{proof}

\begin{lemma}\label{l:2} 
  Suppose $\mathsf A_i,\mathsf B_i$ are finite dimensional \kk-algebras,
  $f_i:\mathsf A_i\to\mathsf B_i$ ($i=1,2$) are algebra morphisms, and
  $f_i^*:\mathsf{B}_i\text-\mathrm{Mod}_f\to
  \mathsf{A}_i\text-\mathrm{Mod}_f$ the corresponding restriction of
  scalars functors. 
  Then $f_1^*\boxtimes f_2^*$ is an exact functor. 
\end{lemma}
\begin{proof}
  Accordingly to Example \ref{ex:1}, $\mathsf{A}_i\text-\mathrm{Mod}_f
  \boxtimes \mathsf{B}_i\text-\mathrm{Mod}_f$ is equivalent to
  $\mathsf{A}_i\otimes\mathsf{B}_i\text-\mathrm{Mod}_f$. The functor 
  $f_1^*\boxtimes f_2^*$ is the unique up to isomorphism right exact
  functor fitting in a diagram as exhibited below.
  \begin{equation*}
    \diagram
    \mathsf{A_1}\text-\mathrm{Mod}_f\otimes
    \mathsf{B_1}\text-\mathrm{Mod}_f \ar[r]^-{\otimes_{\kk}}
    \ar[d]_{f_1^*\otimes f_2^*}
    &
    \mathsf{A_1}\otimes\mathsf{B_2}\text-\mathrm{Mod}_f
    \ar[d]^{f_1^*\boxtimes f_2^*}
    \\
    \mathsf{A_2}\text-\mathrm{Mod}_f\otimes
    \mathsf{B_2}\text-\mathrm{Mod}_f \ar[r]^-{\otimes_{\kk}}
    &
    \mathsf{A_2}\otimes\mathsf{B_2}\text-\mathrm{Mod}_f
    \ar@{}[ul]|-\cong
    \enddiagram
  \end{equation*}
  Since $(f_1\otimes f_2)^*$ clearly satisfies that property, there is
  a natural isomorphism between this restriction of scalars functor
  and $f_1\boxtimes f_2$. The result follows from the
  fact that restriction of scalars functors are exact.  
\end{proof}

  Recall from Section \ref{sec:bicolimits} the notion of a bicolimit
  of a pseudofunctor. Here we will be interested in colimits of
  pseudofunctors $F:J\to \kk\text-\mathrm{Mod}\text-\Cat$ from
  a filtered category $J$.
  By giving an explicit construction of $\operatorname{bicolim}F$ the
  following can be proved. 
  \begin{lemma}
    \label{l:15}
    Let $J$ be a filtered category.
    If for each arrow $\alpha:j\to k$ in $J$ the functor
    $F\alpha:Fj\to Fk$ is right exact (resp. exact) between finitely
    cocomplete (resp. abelian) categories, then
    $\operatorname{bicolim}F$ is finitely cocomplete (resp. abelian)
    and the coprojections $Fj\to\operatorname{bicolim}F$ are right
    exact (resp. exact).
  \end{lemma}
  Always assuming that $J$ is filtered, when $F:J\to
  (\kk\text-\mathrm{Mod})\text-\Cat$ is not only a 
  pseudofunctor but a 2-functor, then its colimit (in the usual strict
  sense)
  is also a bicolimit. This last fact appears in \cite[Expose VI, Exercice 6.8,
  p. 272]{SGA4}.

\begin{proof}[Proof of Proposition \ref{prop:4}]
  According to Lemma \ref{l:7} we can write $A$ and $B$ as colimit of
  filtered of diagrams of finitely cocomplete categories and right exact
  functors, $A=\colim_\alpha A_\alpha$ and $B=\colim_\beta
  B_\beta$, where each $A_\alpha,B_\beta$ is equivalent to a category
  of finite dimensional modules over a finite dimensional algebra.
  In view of
  Theorem \ref{thm:3}, we need to show that $A\boxtimes B$ is
  abelian. As remarked at the end of Section \ref{sec:tens-prod-finit},
  $\boxtimes$ preserves bicolimits in each variable, so
  \begin{equation}
    \label{eq:9}
    A_\alpha\boxtimes B_\beta\to A\boxtimes B
  \end{equation}
  is a bicolimit of the pseudofunctor $J\times J\to\mathbf{Rex}$ given by
  $(\alpha,\beta)\mapsto A_\alpha\boxtimes B_\beta$.
  By Example \ref{ex:1} and Lemma
  \ref{l:2}, for each $(\alpha,\beta)\to(\alpha',\beta')$ in $J\times
  J$, each $A_\alpha\boxtimes B_\beta\to A_{\alpha'}\boxtimes
  B_{\beta'}$ is an exact functor between abelian categories, an so
  $A\boxtimes B$ is abelian by Lemma \ref{l:15}.   
\end{proof}

\section{Semisimple categories}
\label{sec:semis-categ}

In this section we study in some detail tensor products of semisimple
$k$-linear abelian categories satisfying an extra finiteness condition, namely
each object has finite length. This particular case is of interest as
it includes fusion categories \cite{Etiongof:OnFusionCats}.
It is well known that, when $k$ is
algebraically closed, the Deligne tensor product of two such categories $A,B$
exists and is a semisimple category with simple objects corresponding
to pairs of simple objects $(a,b)$ of $A$ and $B$ respectively. The
results of this section show that this easy description
is a consequence of the fact that semisimple abelian categories
with objects of finite length are free, in the appropriate sense.

Recall from Section \ref{sec:compl-under-colim} that the completion of
$X$ under finite colimits can be described as the smallest replete full
subcategory $\fin(X)\subset[X^\op,\V]$ closed under finite colimits
and containing the representable presheaves.

\begin{prop}
  \label{prop:2}
  For a \kk-linear category $A$ the following are equivalent.
  \begin{enumerate}
  \item
    $A$ is semisimple abelian with objects of finite length.
  \item
    $A$ is equivalent to $\fin(X)$ where the category $X$ is a
    coproduct in \VCat\ of division \kk-algebras.
  \end{enumerate}
  In this case, $X$ is (a skeleton of) the full subcategory of simple
  objects.
\end{prop}
\begin{proof}
  There are a number of possible proofs of this proposition, the
  following one being a relatively straightforward one. 
  Suppose $A$ is semisimple abelian with objects of finite length, and let $X$
  be (a skeleton of) the full subcategory of simple objects, with
  inclusion $K:X\hookrightarrow{} A$. It is easy to see that $X$ is a
  strong generator because every object of $A$ is isomorphic to a
  direct sum of objects of $X$. Every object in $A$ is projective, so
  $A\simeq \fin(X)$ ---see Section \ref{sec:finite-colimits}.

  To prove the converse, write $X\cong \sum_\alpha X_\alpha$, where
  $X_\alpha$ is a category with one object and one hom that is a
  division \kk-algebra $\mathsf A_\alpha$. Each category
  $\mathsf A_\alpha\text-\mathrm{Mod}$ is semisimple, and so is their product
  $\prod_\alpha   \mathsf A_\alpha\text-\mathrm{Mod}$; there is a
  simple object $s(\alpha)$ for each $\alpha$, with
  $s(\alpha)_\beta=0$ if $\alpha\neq\beta$ and
  $s(\alpha)_\alpha=\mathsf A_\alpha$.
  Each one of these simple objects correspond to a representable
  presheaf under the following equivalence.   
  \begin{equation*}
    [X^\op,\V]\cong \prod_\alpha[X_\alpha^\op,\V]\simeq \prod_\alpha
    \mathsf A_\alpha\text-\mathrm{Mod}
  \end{equation*}
  Therefore, $[X^\op,\V]$ is semisimple abelian and the simple objects are the
  representables. 
  By definition $\fin(X)$
  is the smallest replete full subcategory closed under colimits and
  containing the representables, so it is semisimple and it consists
  of the finite direct sums of simple objects.  
\end{proof}

Observe that in the above proposition we also proved that such a
category $A$ is the free completion of $X$ under absolute colimits,
and also the free completion of $X$ under finite direct sums. 

\begin{thm}
  \label{thm:2}
  Suppose $A,B$ are two semisimple abelian categories with objects of
  finite length, and assume the endo-hom of each simple object is the
  ground commutative ring \kk.
  \begin{enumerate}
  \item \label{item:18}
    Then $A\boxtimes B$ satisfies these
    same properties; in particular Deligne's tensor product of $A,B$
    exists and it is (equivalent to) $A\boxtimes B$. 
  \item\label{item:19}
    The simple objects in $A\boxtimes B$ are (up to isomorphism) the
    images of pairs of simple objects
    under the universal functor $A\otimes B\to A\boxtimes B$.
  \item\label{item:20}
    The universal functor $A\otimes B\to A\boxtimes B$ is fully
    faithful. 
\end{enumerate}
\end{thm}
\begin{proof}
  $A,B$ are equivalent to $\fin(X),\fin(Y)$ where $X,Y$ are {\em
    discrete}\/ $\kk$-linear categories. We have $A\boxtimes B\simeq \fin(X\otimes
  Y)$, and $X\otimes Y$ is a discrete category too. By Proposition
  \ref{prop:2} 
  $A\boxtimes B$ is semisimple abelian with objects of finite length,
  and each simple object has endo-hom \kk. Deligne's tensor product of
  $A,B$ exists by Theorem \ref{thm:3}.
  The diagram below, that
  commutes up to isomorphism, together with Proposition \ref{prop:2}
  proves the claim \ref{item:19}.
  \begin{equation*}
    \xymatrixrowsep{.5cm}
    \diagram
    X\otimes Y\ar[r]\ar[rrd]&
    \fin(X)\otimes\fin(Y)\ar[r]\ar@{}[dr]^\cong&
    \fin(X)\boxtimes\fin(Y)\ar[d]^\simeq\\
    &&\fin(X\otimes Y)
    \enddiagram
  \end{equation*}
  To show \ref{item:20} we apeal to Remark \ref{rmk:6}.
\end{proof}

\begin{cor}
  \label{cor:1}
  Assume \kk\ is an algebraically closed field. Then the tensor
  product $A\boxtimes B$ of two semisimple abelian categories with
  finite dimensional homs and 
  objects of finite length $A,B$ has these same properties.
\end{cor}
\begin{proof}
  The algebraic closedness of \kk\ ensures that the \kk-algebra of
  endomorphisms of a  simple object, if finite dimensional, is
  isomorphic to \kk. 
\end{proof}

\bibliographystyle{abbrv}
\bibliography{deligne_only}  
\end{document}